\documentclass[reqno]{amsart}
\usepackage{amssymb}
\usepackage{mathrsfs}
\usepackage{amsmath,amstext,amsfonts,verbatim}
\usepackage[numbers,sort&compress]{natbib}

\usepackage{hyperref} 
 \hypersetup{colorlinks,
  linkcolor=blue,%
  citecolor=blue}

\usepackage[latin1]{inputenc}

\usepackage{amsmath,amsthm,amssymb}

\usepackage{amsfonts}

\usepackage{amsmath,amssymb}
\usepackage{graphicx}
\usepackage{color}
\usepackage{indentfirst}

\newtheorem{theorem}{Theorem}[section]
\newtheorem{lemma}[theorem]{Lemma}

\theoremstyle{definition}
\newtheorem{definition}[theorem]{Definition}

\theoremstyle{remark}
\newtheorem{remark}[theorem]{Remark}

\numberwithin{equation}{section}

\newcommand{\neweq}[1]{\begin{equation}\label{#1}}

\def\phi{\varphi}

\def\incep{\left\{\begin{array}{cl} }
 \def\termin{\end{array}\right. }
\def\2af{2^*_\alpha}

\begin{document}

\title[Quasihyperbolic metric and Gromov hyperbolic spaces I]
{\textbf{Quasihyperbolic metric and Gromov hyperbolic spaces I}}

\author{Hongjun Liu}
\address{Hongjun Liu, School of Mathematical Sciences,
Guizhou Normal University, Guiyang, 550025, China}
\email{hongjunliu@gznu.edu.cn}

\author{Ling Xia}
\address{Ling Xia, School of Mathematical Sciences,
Guizhou Normal University, Guiyang, 550025, China}
\email{18230980941@163.com}

\author{Shasha Yan}
\address{Shasha Yan, School of Mathematical Sciences,
Guizhou Normal University, Guiyang, 550025, China}
\email{3072269500@qq.com}

\thanks{Supported by National Natural Science Foundation of China (Grant No. 12461012) and
the Guizhou Province Science and Technology Foundation (Grant No. QianKeHeJiChu[2020]1Y003.}




\begin{abstract}
In this paper, we introduce the concepts of short arc and length map in quasihyperbolic metric spaces,
and obtain some geometric characterizations of Gromov hyperbolicity for quasihyperbolic metric spaces in terms of
the properties of short arc and length map.

\end{abstract}

\maketitle

\section{Introduction and main results}

The quasihyperbolic metric was introduced by Gehring and his students
Palka and Osgood in the 1970's \cite{GO, GP} in the setting of Euclidean spaces $\mathbb{R}^n (n \geq 2)$.
Since its first appearance, the quasihyperbolic metric
has become an important tool in geometric function theory
and in its generalizations to metric spaces and Banach spaces \cite{BHK, V05-1, Ge62, HL15, V99, KVZ, ZLH21}.
Recently, based on the quasihyperbolic metric, V\"{a}is\"{a}l\"{a}
developed a ``dimension-free" theory of quasiconformal mappings
in infinite-dimensional Banach space and obtained many beautiful results.
More relevant literature can be founded in \cite{HK98, HL15, TV, V90, V91, V92, V98, V99}.

The Gromov hyperbolicity is a concept introduced by Gromov in the setting of geometric group theory in 1980's \cite{Gro}.
Loosely speaking, this property means that a general metric space is ``negatively curved",
in the sense of the coarse geometry. The concept generalizes a fundamental property of the hyperbolic metric,
which has a constant negative Gaussian curvature. Unlike the Gaussian curvature,
which is dependent on the two dimensional surface theory,
the concept of the Gromov hyperbolicity is applicable in a wide range of metric spaces.
Because much of the classical geometric function theory
can be built on the study of the geometry of the hyperbolic metric \cite{Wo, Be},
it is hardly surprising that there are many connections to the generalizations of
the other concepts originating from the classical function theory as well.

Indeed, since its introduction, the theory of the Gromov hyperbolicity has been found numerous applications,
and it has been, for example, considered in the books \cite{BB99, Bon96, BHK, BS00, ZPG22, ZLL22}.
Initially, the research was mainly focused on the hyperbolic group theory \cite{KB02, Bow91}.
Recently, researchers have shown an increasing interest in the study of the Gromov hyperbolicity from different points of view.
For example, geometric characterizations of the Gromov hyperbolicity have been established in \cite{Ba, ZLL21}.
In \cite{BHK}, Bonk et al. investigated negative curvature of uniform metric spaces
and demonstrated many phenomena in function theory from the point of view of Gromov hyperbolicity of the quasihyperbolic metric.
Furthermore, the close connection between the Gromov hyperbolicity and quasiconformal deformations has been studied in \cite{BS00, Ge62, GP, HK98, ZP23}.
The Gromov hyperbolicity of various metrics and surfaces has been investigated in \cite{ZLL21, ZPG22}.
For other discussions in this line see, for example, \cite{BB99, BS00, Kos, Zq, V05-2}.

The main aim of the present paper is to define the Gromov product in the the quasihyperbolic metric,
and then discuss the properties of Gromov hyperbolic on the quasihyperbolic metric spaces with the aid of quasihyperbolic short arcs.
Following analogous notations and terminologies of
\cite{BB99, Bon96, BHK, Ba, BS00, Bow91, Kos, Zq, V05-1, V05-2, ZP23, ZLL21, ZLL22, ZPG22, V05},
now we formally introduce the definitions of Gromov hyperbolicity in the quasihyperbolic metric and some related concepts.

\begin{definition}
Let $(X, k_X)$ be a quasihyperbolic metric space,
fix a base point $\omega\in X$, for any two points $x, y \in X$,
we define
$$
(x|y)_{\omega}=\frac{1}{2}\big(k_X(x,\omega)+k_X(y,\omega)-k_X(x,y)\big).
$$
This number is called the {\it Gromov product} of $x$ and $y$ with respect  to $\omega$.

We say that $(X, k_X)$ is {\it Gromov $\delta$-hyperbolic space},
if there is a constant $\delta\geq 0$ such that
$$
(x|y)_{\omega} \geq \min\{(x|z)_{\omega},(z|y)_{\omega}\}-\delta
$$
for all $x, y, z, \omega \in X$.

If $(X, k_X)$ is Gromov $\delta$-hyperbolic for some $\delta\geq 0$,
then $(X, k_X)$ is called {\it Gromov hyperbolic space}.
\end{definition}

The quasihyperbolic metric $k_X(\cdot, \cdot)$ of $X$ was introduced by Gehring and Osgood in \cite{GO}.
For the concepts of quasihyperbolic metric $k_X(\cdot, \cdot)$, please see Section 2.
It is well known that $(X, k_X)$ is complete, proper, and geodesic as a metric space, see (\cite{BHK}, Proposition 2.8).

\begin{definition}
Let $X$ be a Gromov $\delta$-hyperbolic space in the quasihyperbolic metric for some constant $\delta \geq 0$,
and fix a base point $\omega\in X$.
\begin{enumerate}
\item  $A$ sequence of points $\{x_{i}\}\subseteq X$ is said to a {\it Gromov sequence}
if $(x_{i}|x_{j})_{\omega} \rightarrow \infty$ as $i,j \rightarrow \infty$;
\vspace{0.1cm}
\item Two such sequences $\{x_{i}\}$ and $\{y_{i}\}$ are said to be {\it equivalent}
if $(x_{i}|y_{i})_{\omega} \rightarrow \infty$ as $i \rightarrow \infty$;
\vspace{0.1cm}
\item The {\it Gromov boundary} or {\it the boundary at infinity} $\partial_{\infty}X$ of $X$
is defined to be the set of all equivalent classes of {\it Gromov sequences}, and $X^*=X\cup \partial_{\infty}X$
is called the {\it Gromov closure} of $X$;
\vspace{0.1cm}
\item For $a \in X$ and $\eta \in \partial _{\infty}X$,
the {\it Gromov} product $(a|\eta)_{\omega}$ of $a$ and $\eta$ is defined by
$$
(a|\eta)_{\omega}=\inf \big\{\liminf_{i\rightarrow \infty} (a|b_{i})_{\omega}:\{b_{i}\}\in \eta \big\};
$$

\item For $\xi,\eta \in \partial _{\infty}X$,
the {\it Gromov} product $(\xi|\eta)_{\omega}$ of $\xi$ and $\eta$ is defined by
$$
(\xi|\eta)_{\omega}=\inf \big\{\liminf_{i\rightarrow \infty} (a_{i}|b_{i})_{\omega}: \{a_{i}\}\in \xi
\quad {\text {and}}\quad \{b_{i}\}\in \eta \big\}.
$$
\end{enumerate}
\end{definition}

\begin{remark}\label{G-r}
It is easy to prove that whether a sequence $\{x_{i}\}$ is a Gromov sequence
or not is independent of the subscript $\omega$ of Gromov product $(x_{i}|x_{j})_{\omega}$
and the equivalence of two Gromov sequences is essentially an equivalence relation.
\end{remark}

In 2005, V\"{a}is\"{a}l\"{a} studied the properties of Gromov hyperbolic spaces and obtained a series of novel results.
Under suitable geometric conditions (see Section 2),
in this paper we shall prove a more general result (Theorem \ref{main}) for quasihyperbolic metric spaces.
Our proof is based on a refinement of the method due to V\"{a}is\"{a}l\"{a} \cite{V05}.

\begin{theorem}\label{main}
Let $X$ be a Gromov $\delta$-hyperbolic space in the quasihyperbolic metric,
and let $a, b\in\partial_{\infty}X$, $a\neq b$, $h>0$.
Then there is a sequence $\bar{\alpha}$ of arcs $\alpha_{i}:x_{i} \curvearrowright y_{i}$ in $X$,
where $\{x_{i}\}\in a$ and $\{y_{i}\}\in b$,
which satisfies the following properties:
\begin{enumerate}
\item Each $\alpha_i$ is quasihyperbolic $h$-short arc;
\vspace{0.1cm}
\item For some $s_{1} \in \alpha_{1}$, we have $k_{X}(x_{i},s_{1})\rightarrow\infty$
      and $k_{X}(y_{i},s_{1})\rightarrow\infty$ as $i\rightarrow\infty$;
\vspace{0.1cm}
\item For $i \leq j\leq m$, there is length map $f_{ij}:\alpha_{i}\rightarrow \alpha_{j}$ with $f_{ii}=id$, $ f_{im}=f_{jm}\circ f_{ij}$;
\vspace{0.1cm}
\item $k_{X}(f_{ij}(u),u)\leq 12(\delta+h)$ for all $u \in \alpha_{i}$ and $i\leq j$.
\end{enumerate}
\end{theorem}

\begin{remark}
For the concepts of length map and quasihyperbolic $h$-short arc, please see Section 2.
\end{remark}

\vspace{0.1cm}

The rest of this paper is organized as follows.
In Section 2, we will introduce some necessary notations and concepts,
recall some known results, and prove a series of basic and useful results.
We mainly demonstrate in Section 3 that some geometric characterizations of
Gromov hyperbolicity for quasihyperbolic metric spaces.

\vspace{0.1cm}


\section{Preliminaries}

In this section, we give the necessary definitions and auxiliary results, which will be used in the proofs
of our main results.

\vspace{0.1cm}

\medskip\noindent{\bf{2.1. Quasihyperbolic metric}}

In this subsection we review some basic definitions and facts about quasihyperbolic metric,
see \cite{GO, KVZ, LLYT,ZP23, HL15, V90, V99, ZLH21}, and references therein for more details.

Let $X$ be a metric space.
By a {\it curve} we mean any continuous mapping $\gamma : [a,b]\rightarrow X$.
The {\it length} of $\gamma$ is defined by
$$
l(\gamma)=\sup\left\{
\sum_{i=1}^{n}\big|\gamma(t_i)-\gamma(t_{i-1})\big|\right\},
$$
where the supremum is taken over all partitions $a=t_0< t_1< \cdots < t_n=b$ of the interval $[a,b]$.
If $l(\gamma)<\infty$, then the curve $\gamma$ is said to be a {\it rectifiable curve}.

Let $\gamma\subset X$ be a rectifiable curve of length $l(\gamma)$ with endpoints $a$ and $b$.
The {\it {length function}} associated with a rectifiable curve
$\gamma : [a, b]\rightarrow X$ is $s_{\gamma}: [a, b] \rightarrow [0, l(\gamma)]$,
given by $s_{\gamma}(t)=l \big(\gamma|_{[a, t]}\big)$. For any rectifiable curve $\gamma : [a, b]\rightarrow X$,
there exists a unique curve $\gamma_s : [0, l(\gamma)]\rightarrow X$
such that $\gamma=\gamma_s\circ s_{\gamma}$. The curve $\gamma_s$ is called
the {\it arc length parametrization} of $\gamma$. This parametrization is characterized by the relation
$l \big(\gamma_s|_{[t_1,t_2]}\big)=t_2-t_1$ for all $0\leq t_1<t_2\leq l(\gamma)$.

If $\gamma$ is a rectifiable curve in $X$, the line integral over
$\gamma$ of each nonnegative Borel function $\varrho : X\rightarrow [0, \infty]$ is given by
$$
\int_{\gamma}\varrho\,
ds=\int_{0}^{l(\gamma)}\varrho\circ\gamma_{s}(t)\, dt.
$$

Let $\gamma$ be a rectifiable curve in metric space $X$.
The {\it quasihyperbolic length} of $\gamma$ in $X$ is the number
$$
l_{k_{X}}(\gamma)=\int_{\gamma}\frac{ds}{\delta_{X}(x)},
$$
for each $x\in X$, where $\delta_{X}(x)$ denote the distance ${\text {dist}}(x, \partial X)$.
For each pair of points $x,y\in X$, the {\it quasihyperbolic distance}
$k_X(x,y)$ between $x$ and $y$ is defined by
$$
k_{X}(x, y)=\inf_{\gamma}\, l_{k_{X}}(\gamma),
$$
where $\gamma$ runs over all rectifiable curves in $X$ joining $x$ and $y$.
If there is no rectifiable curve in $X$ joining $x$ and $y$,
then we define $k_{X}(x, y)=+\infty$.

A metric space $X$ is said to be {\it rectifiably connected} if for any two points $x, y\in X$,
there exists a rectifiable curve in $X$ joining $x$ and $y$.
If $X$ is a rectifiably connected open set,
it is clear that $k_X(x, y)<\infty$ for any two points $x, y\in X$.
Thus it is easy to verify that $k_X(\cdot,\cdot)$ is a metric in $X$,
and we called the {\it {quasihyperbolic metric}} of $X$,
then $(X, k_X)$ is a quasihyperbolic metric space.
Throughout the article, we let $X$ denote a quasihyperbolic metric space.

\vspace{0.1cm}

\medskip\noindent{\bf{2.2. Quasihyperbolic short arc and length map}}

By a space we mean a quasihyperbolic metric space.
The distance between points $x$ and $y$ is usually written as $k_{X}(x,y)$.
An arc in a metric space $X$ is a subset homeomorphic to a real interval.
Unless otherwise stated, this interval is assumed to be closed. Then the arc
is compact and has two endpoints. We write $\alpha : x \curvearrowright y$ if $\alpha$
is an arc with endpoints $x$ and $y$.
For an arc $\alpha$ , we let $\alpha|_{[u, v]}$ denote the closed subarc of $\alpha$ between points $u, v \in \alpha$.

\begin{definition}
Let $h \geq 0$. We say that an arc $\gamma: x \curvearrowright y$  is {\it quasihyperbolic $h$-short arc} in $X$, if
$$
l_{k_{X}}(\gamma)\leq k_X(x,y)+h.
$$
\end{definition}

\begin{remark}\label{re-2}
For all $x,y\in X$, from the definition of $k_{X}(x, y)$, it follows that
$$
l_{k_X}(\gamma)\leq k_{X}(x,y)+\varepsilon
$$
for all $\varepsilon>0$.
That is, for each pair $x,y\in X$ and for each $h>0$, there is a quasihyperbolic $h$-short arc
$\alpha: x\curvearrowright y$.
\end{remark}

\begin{definition}
Let $X$ be a quasihyperbolic metric space.
Suppose that $\alpha$, $\beta$ are rectifiable arcs in $X$ with $l_{k_X}(\alpha)\leq l_{k_X}(\beta)$.
A map $f:\alpha \rightarrow \beta$  is said to be {\it length map} if
$$
l_{k_X}\big(\beta|_{[f(u),f(v)]}\big)=l_{k_X}\big(\alpha|_{[u,v]}\big)
$$
for all $u,v\in \alpha$.
\end{definition}

\begin{theorem}\label{qc-1}
Let $X$ be a quasihyperbolic metric space.
Suppose that $\gamma:x\curvearrowright y$ is quasihyperbolic $h$-short arc in $X$.
Then
\begin{equation}\label{a}
\begin{split}
k_{X}(x, z)-\frac{h}{2}\leq (z|y)_{x}\leq k_{X}(x, z)
\end{split}
\end{equation}
for all $z\in \gamma$.
\end{theorem}
\begin{proof}
Let $h\geq 0$. For all $z\in \gamma$,
according to the definitions of quasihyperbolic metric and quasihyperbolic $h$-short arc,
it follows that
\begin{equation*}\label{2.1}
\begin{split}
k_{X}(x,z)+k_{X}(z,y) &\leq l_{k_{X}}\big(\gamma|_{[x,z]}\big)+l_{k_{X}}\big(\gamma|_{[z,y]}\big)\\
&=l_{k_{X}}(\gamma)\leq k_{X}(x, y)+h.
\end{split}
\end{equation*}
Therefore,
$$
2k_{X}(x,z)-h\leq k_{X}(x,z)+k_{X}(x,y)-k_{X}(z, y)=2(z|y)_{x}.
$$
Hence, we have
$$
(z|y)_{x}\geq k_{X}(x,z)-\frac{h}{2}.
$$
Now, we prove the right inequality of the inequality (\ref{a}).
By the definition of Gromov product, we obtain that
\begin{equation*}
\begin{split}
2(z|y)_{x} &= k_{X}(x,z)+k_{X}(x,y)-k_{X}(y,z)\\
& \leq k_{X}(x,z)+k_{X}(x,y)-\big(k_{X}(x,y)-k_{X}(x,z)\big)\\
&=2k_{X}(x,z),
\end{split}
\end{equation*}
that is, $(z|y)_{x}\leq k_{X}(x, z).$
Hence, This proof is completed.
\end{proof}

\begin{theorem}\label{qc-2}
Let $X$ be a quasihyperbolic metric space.
Suppose that $\gamma: p\curvearrowright q$ is a quasihyperbolic $h$-short arc in $X$,
and that $y_{1}, y_{2}\in \gamma$. Then the subarc $\gamma|_{[y_{1}, y_{2}]}$ is quasihyperbolic $h$-short arc.
\end{theorem}
\begin{proof}
Let $h\geq 0$ and $y_{1}, y_{2}\in \gamma$.
From the definition of quasihyperbolic $h$-short arc, we have
$$
l_{k_{X}}(\gamma)\leq k_{X}(p, q)+h.
$$
From the above inequality,
it follows immediately from the definition of quasihyperbolic distance that
\begin{equation*}
\begin{split}
l_{k_{X}}\big(\gamma|_{[y_{1},y_{2}]}\big)&=l_{k_{X}}(\gamma)-l_{k_{X}}\big(\gamma|_{[p,y_{1}]}\big)-l_{k_{X}}\big(\gamma|_{[y_{2},q]}\big)\\
&\leq l_{k_{X}}(\gamma)-k_{X}(p, y_{1})-k_{X}(y_{2},q)\\
&\leq k_{X}(p, q)+h -k_{X}(p, y_{1})-k_{X}(y_{2}, q).
\end{split}
\end{equation*}
Since $k_{X}(p, q)-k_{X}(p, y_{1})-k_{X}(y_{2}, q)\leq k_{X}(y_{1},y_{2})$,
we have
$$
l_{k_{X}}\big(\gamma|_{[y_{1}, y_{2}]}\big)\leq k_{X}(y_{1}, y_{2})+h.
$$
Hence, the subarc $\gamma|_{[y_{1}, y_{2}]}$ is quasihyperbolic $h$-short arc.
\end{proof}

\vspace{0.1cm}

\section{Proof of Theorem \ref{main}}

In what follows, we always assume that $X$ is a Gromov $\delta$-hyperbolic space
in the quasihyperbolic metric, $a\in\partial_{\infty} X$, $h>0$.
To prove Theorem \ref{main},
we divide this section into two subsections. In the first subsection,
a useful lemma will be proved, and the proof of Theorem \ref{main}
will be presented in the second subsection.

\medskip\noindent{\bf{3.1. An auxiliary result.}}
\vspace{0.1cm}

The following result plays a key role in the proof of Theorem \ref{main}.
Based on \cite{V05-2}, for the completeness, we will prove the following result.

\begin{lemma}\label{le-1}
Let $X$ be a Gromov $\delta$-hyperbolic space in the quasihyperbolic metric,
and let $x\in X$, $a\in\partial_{\infty} X$, $h>0$.
Then there is a sequence $\bar{\beta}$ of arcs $\beta_{m}: x \curvearrowright x_{m}$ and $\{x_{m}\}\in a$,
which satisfies the following properties:
\begin{enumerate}
\item Each $\beta_{m}$ is quasihyperbolic $h$-short;
\vspace{0.1cm}

\item  The sequence of quasihyperbolic lengths $l_{k_{X}}(\beta_{m})$ is increasing and tends to $\infty$;
\vspace{0.1cm}

\item  For $m \leq n$, the length map $f_{mn}:\beta_{m}\rightarrow \beta_{n}$ with $f_{mn}(x)=x$ satisfies
$$
k_{X}\big(f_{mn}(z),z\big)\leq 4\delta+2h   \quad {\text {for all}}\quad  z \in \beta_{m}.
$$
\end{enumerate}
\end{lemma}

\begin{proof}
The assumption implies that $X$ is a Gromov $\delta$-hyperbolic space in the quasihyperbolic metric,
and that $x\in X$, $a\in\partial_{\infty} X$, $h>0$. Let $\{u_{i}\}\in a$,
it follows from the definition of Gromov boundary that $\{u_{i}\}$ is a Gromov sequence.
According to Remark \ref{re-2}, for any two points in $X$, there is a quasihyperbolic $h$-short arc connecting them,
then we let $\alpha_{N}:x \curvearrowright u_{N} $ be quasihyperbolic $h$-short arc, where $N\in \mathbb{N}_{+}$.
Since $\{u_{i}\}$ is a Gromov sequence in $X$,
we can get $(u_{i}|u_{j})_{x}\rightarrow \infty$ as $i,j\rightarrow \infty$,
that is, for any $m >0$, there exists $N(m)\in \mathbb{N}_{+}$ such that whenever $i,j\geq N(m)$, then
\begin{equation}\label{le-1.1}
\begin{split}
(u_{i}|u_{j})_{x} \geq m
\end{split}
\end{equation}
and such that
$$N(1)< N(2)< N(3)<\cdots$$
By inequality (\ref{le-1.1}), it follows that
$$
l_{k_{X}}\big(\alpha_{N(m)}\big)\geq k_{X}\big(x,u_{N(m)}\big)=\big(u_{N(m)}|u_{N(m)}\big)_{x}\geq m.
$$
Therefore, there exists a subarc $\beta_{m}=\alpha_{N(m)}|_{[x,x_{m}]}$ such that
\begin{equation}\label{le-1.2}
\begin{split}
l_{k_{X}}(\beta_m)=m,
\end{split}
\end{equation}
where $x_{m}\in \alpha_{N(m)}$.

Hence, (1) follows from Theorem \ref{qc-2}.
Part (2) follows directly from inequality (\ref{le-1.2}).

In what follows, we only need to prove $\{x_{m}\}\in a$ and (3).

\vspace{0.1cm}
\noindent {\bf {Claim 3.1.}}\,\, $\{x_{m}\}$ is a Gromov sequence.
\vspace{0.1cm}

In order to prove $\{x_{m}\}$ is a Gromov sequence,
it suffices to prove $(x_{m}|x_{n})_{x}\rightarrow\infty$ as $m,n\rightarrow\infty$.

Applying the conclusion (1) of Lemma \ref{le-1}, we know that $\beta_{m}$ is quasihyperbolic $h$-short arc.
In view of the fact $x_{m}\in \alpha_{N(m)}$ and $\alpha_{N(m)}:x \curvearrowright u_{N(m)} $ is a quasihyperbolic $h$-short arc,
by Theorem \ref{qc-1} and inequality (\ref{le-1.2}), we have
\begin{eqnarray*}
\begin{split}
\big(x_{m}|u_{N(m)}\big)_{x}&\geq k_{X}(x,x_{m})-\frac{h}{2}\\&\geq  l_{k_{X}}(\beta_{m})-h-\frac{h}{2}\\&=m-\frac{3h}{2},
\end{split}
\end{eqnarray*}
which implies that
\begin{equation}\label{le-1.3}
\begin{split}
\big(x_{m}|u_{N(m)}\big)_{x}\rightarrow\infty   \quad {\text {as}}\quad  m\rightarrow\infty.
\end{split}
\end{equation}

In addition, by $N(m)\in \mathbb{N}_{+}$ and $N(m)< N(m+1)<\cdots$,
we get that
\begin{equation}\label{le-1.4}
\begin{split}
N(m)\rightarrow\infty  \quad {\text {as}}\quad m\rightarrow\infty.
\end{split}
\end{equation}
Since $\{u_{N(m)}\}$ is a subsequence of Gromov sequence $\{u_{i}\}$,
it is obviously a Gromov sequence equivalent to $\{u_{i}\}$.
Thus, by (\ref{le-1.4}), we have
\begin{equation}\label{le-1.4-1}
\begin{split}
\big(u_{N(m)}|u_{N(n)}\big)_{x}\rightarrow\infty    \quad {\text {as}}\quad  m, n\rightarrow\infty
\end{split}
\end{equation}
By the definition of the Gromov hyperbolic space, it is clear that
$$
(x_{m}|x_{n})_{x}\geq \min\left\{\big(x_{m}|u_{N(m)}\big)_{x},\big(u_{N(m)}|u_{N(n)}\big)_{x},\big(u_{N(n)}|x_{n}\big)_{x}\right\}-2\delta.
$$
Therefore, combining this inequality with (\ref{le-1.3}) and (\ref{le-1.4-1}), it follows immediately that
$$
(x_{m}|x_{n})_{x}\rightarrow\infty \quad {\text {as}}\quad  m, n\rightarrow\infty.
$$
Hence, $\{x_{m}\}$ is a Gromov sequence. Claim 3.1 is proved.

Moreover, we know from (\ref{le-1.3}) that
$\{x_{m}\}$ is equivalent to $\{u_{N(m)}\}$.
Since $\{u_{N(m)}\}$ is equivalent to $\{u_{i}\}$, where $\{u_{i}\}\in a\in\partial_{\infty} X$,
and since the equivalence of two Gromov sequences is essentially an equivalence relation,
from the properties of equivalence relation and the definition of $\partial_{\infty} X$,
we obtain $\{x_{m}\}\in a$.

At last, we need to show (3). Since the sequence of lengths $l_{k_{X}}(\beta_{m})$ is increasing,
let $m \leq n$, we can set $f_{mn}:\beta_{m} \rightarrow \beta_{n}$ be the length map fixing $x$,
that is $f_{mn}(x)=x$.
In what follows, for all $z \in \beta_{m}$, we want to find an estimate
\begin{equation*}\label{le-1.5}
\begin{split}
k_{X}\big(f_{mn}(z),z\big)\leq 4\delta+2h.
\end{split}
\end{equation*}

Since $\alpha_{N(m)}:x \curvearrowright u_{N(m)}$ and $\alpha_{N(n)}:x \curvearrowright u_{N(n)}$
are quasihyperbolic $h$-short arcs in Gromov hyperbolic space, for $z \in \beta_{m}\subset \alpha_{N(m)}$, by (\ref{le-1}) and (\ref{le-1.2}), we have
\begin{equation}\label{le-1.13}
\begin{split}
k_{X}(x,z)\leq l_{k_{X}}\big(\beta_{m}|_{[x, z]}\big)\leq l_{k_{X}}(\beta_{m})=m\leq \big(u_{N(m)}|u_{N(n)}\big)_{x}.
\end{split}
\end{equation}
And since $f_{mn}:\beta_{m} \rightarrow \beta_{n}$  is the length map fixing $x$, it follows that
$$
l_{k_{X}}\big(\beta_{m} |_{[x, z]}\big)=l_{k_{X}}\big(\beta_{n}|_{[x, f_{mn}(z)]}\big).
$$
Therefore,
\begin{equation*}\label{le-1.14}
\begin{split}
l_{k_{X}}\big(\alpha_{N(m)}|_{[x, z]}\big)=l_{k_{X}}\big(\alpha_{N(n)}|_{[x, f_{mn}(z)]}\big).
\end{split}
\end{equation*}

Let $z^{\prime}\in \alpha_{N(n)}$ be a point with
$$
k_{X}(x,z^{\prime})=k_{X}(x,z).
$$
By Theorem \ref{qc-1}, we know that
\begin{equation}\label{le-1.8}
\begin{split}
 (z|u_{N(m)})_{x}\geq k_{X}(x,z)-\frac{h}{2} \quad {\text {and}}\quad
 (z^{\prime}|u_{N(n)})_{x}\geq k_{X}(x,z^{\prime})-\frac{h}{2}.
\end{split}
\end{equation}
From (\ref{le-1.13}), (\ref{le-1.8}) and the definition of Gromov hyperbolic space, we get that
 \begin{eqnarray*}
\begin{split}
k_{X}(x,z)-\frac{k_{X}(z,z^{\prime})}{2}&=(z|z^{\prime})_{x}
\geq \min \big\{(z|u_{N(m)})_{x},\,\, (u_{N(m)}|u_{N(n)})_{x},\,\, (u_{N(n)}|z^{\prime})_{x}\big\}-2\delta\\
&\geq \min \left\{ k_{X}(x,z)-\frac{h}{2},\,\, k_{X}(x,z),\,\, k_{X}(x,z^{\prime})-\frac{h}{2} \right\}-2\delta\\
&=k_{X}(x,z)-\frac{h}{2}-2\delta.
\end{split}
\end{eqnarray*}
Hence, we obtain that
\begin{equation}\label{le-1.10}
\begin{split}
k_{X}(z,z^{\prime})\leq 4\delta+h.
\end{split}
\end{equation}

In addition, by Theorem \ref{qc-2}, it follows that $\alpha_{N(m)}|_{[x,z]}$,
$\alpha_{N(n)}|_{[x,z^{\prime}]}$ and $\alpha_{N(n)}|_{[z^{\prime},f_{mn}(z)]}$ are quasihyperbolic $h$-short arcs.
Hence, we have
\begin{equation}\label{le-1.6}
\begin{split}
k_{X}(x,z)\leq l_{k_{X}}\big(\alpha_{N(m)}|_{[x,z]}\big)\leq k_{X}(x,z)+h,
\end{split}
\end{equation}
and
\begin{equation}\label{le-1.7}
\begin{split}
k_{X}(x,z^{\prime})\leq l_{k_{X}}\big(\alpha_{N(n)}|_{[x,z^{\prime}]}\big)\leq k_{X}(x,z^{\prime})+h.
\end{split}
\end{equation}
Remembering that $l_{k_{X}}\big(\alpha_{N(m)}|_{[x,z]}\big)=l_{k_{X}}\big(\alpha_{N(n)}|_{[x,f_{mn}(z)]}\big)$ implies that
\begin{equation}\label{le-1.11}
\begin{split}
k_{X}(z^{\prime},f_{mn}(z))
&\leq l_{k_{X}}\big(\alpha_{N(n)}|_{[z^{\prime},f_{mn}(z)]}\big)\\
&=\big|l_{k_{X}}\big(\alpha_{N(n)}|_{[x,f_{mn}(z)}\big)-l_{k_{X}}\big(\alpha_{N(n)}|_{[x,z^{\prime}]}\big)\big|\\
&=\big|l_{k_{X}}\big(\alpha_{N(m)}|_{[x,z]}\big)-l_{k_{X}}\big(\alpha_{N(n)}|_{[x,z^{\prime}]}\big)\big|.
\end{split}
\end{equation}
Combing (\ref{le-1.6}), (\ref{le-1.7}) and (\ref{le-1.11}), it follows that
\begin{equation}\label{le-1.12}
\begin{split}
k_{X}(z^{\prime},f_{mn}(z))\leq h.
\end{split}
\end{equation}
According to inequality (\ref{le-1.10}) and (\ref{le-1.12}), we deduce that
$$
k_{X}(f_{mn}(z),z)\leq k_{X}(z,z^{\prime})+k_{X}(z^{\prime},f_{mn}(z))\leq 4\delta+2h.
$$
This proves (3).

Hence, this completes the proof of Lemma \ref{le-1}.
\end{proof}

\vspace{0.1cm}

\medskip\noindent{\bf{3.2. Proof of Theorem \ref{main}.}}
\vspace{0.1cm}

Suppose that $X$ is a Gromov $\delta$-hyperbolic space in the quasihyperbolic metric,
and that $a, b\in\partial_{\infty}X$, $a\neq b$. Let $z\in X$ and $h>0$.
According to the conclusion of Lemma \ref{le-1},
it follows that there are sequences of quasihyperbolic $h$-short arcs $\beta_{i}:z \curvearrowright x_{i}$
and $\gamma_{i}:z \curvearrowright y_{i}$, where $\{x_{i}\}\in a$ and $\{y_{i}\}\in b$,
and the sequences of quasihyperbolic lengths $l_{k_{X}}(\beta_{i})$ and $l_{k_{X}}(\gamma_{i})$ are increasing and tend to $\infty$,
and for $i \leq j$, we let $g_{ij}:\gamma_{i}\rightarrow \gamma_{j}$ and $h_{ij}:\beta_{i}\rightarrow \beta_{j}$
be the length map fixing $z$, and they satisfy
\begin{equation}\label{m-0}
\begin{split}
k_{X}\big(g_{ij}(v),v\big)\leq 4\delta+2h \quad {\text {and}}\quad  k_{X}\big(h_{ij}(t),t\big)\leq 4\delta+2h
\end{split}
\end{equation}
for all $v \in \gamma_{i}$ and $t \in \beta_{i}$.

Since $a, b\in\partial_{\infty}X$ and $a\neq b$, it is clear that $\{x_{i}\}$ and $\{y_{i}\}$ are nonequivalent Gromov sequences.

\vspace{0.1cm}
\noindent {\bf {Claim 3.2.}}\,\, $\big\{(x_{i}|y_{i})_{z}\big\}$ is bounded for all $i$.
\vspace{0.1cm}

We prove this claim by contradiction. Suppose on the contrary that $\big\{(x_{i}|y_{i})_{z}\big\}$ is unbounded. Then there are subsequences $\{x_{i}'\}$ of $\{x_{i}\}$ and $\{y_{i}'\}$ of $\{y_{i}\}$
such that $(x_{i}'|y_{i}')_{z}\rightarrow \infty$ as $i\rightarrow \infty$. By the assumption, we see that $\{x_{i}'\}$ and $\{y_{i}'\}$ are equivalent Gromov sequences.
Since the equivalence of two Gromov sequences is essentially an equivalence relation,
from the transitivity of the equivalence relation and the fact that
a Gromov sequence is clearly equivalent to each of its subsequence,
we obtain that $\{x_{i}\}$ and $\{y_{i}\}$ are equivalent Gromov sequences,
which contradicts the previous conclusion, so the assumption is not valid.
The proof of the Claim 3.2 is now complete.

We are now turning to the proof of Theorem \ref{main}.

For all $i,j$, without loss of generality, according to Claim 3.2, we may assume that
\begin{equation}\label{m-1}
\begin{split}
\big|(x_{i}|y_{i})_{z}-(x_{j}|y_{j})_{z}\big|\leq h
\end{split}
\end{equation}
and
\begin{equation}\label{m-2}
\begin{split}
l_{k_{X}}(\beta_{i+1})\geq l_{k_{X}}(\beta_{i})+3h,\,\,\, l_{k_{X}}(\gamma_{i+1})\geq l_{k_{X}}(\gamma_{i})+3h.
\end{split}
\end{equation}

Since $X$ is the quasihyperbolic metric space and $x_{i}, y_{i}\in X$ for each $i$, by Remark \ref{re-2},
there exists quasihyperbolic $h$-short arc $\alpha_{i}$ connecting $x_{i}$ and $y_{i}$, which satisfies (1).

To prove (2). Let $s_{1}\in\alpha_{1}$, $h>0$.
From the definition of $\alpha_{1}$, we have
$$
k_{X}(s_1,x_1)\leq l_{k_{X}}\big(\alpha_{1}|_{[x_1,s_1]}\big)\leq l_{k_{X}}(\alpha_{1}).
$$
From the above inequality, it follows that
\begin{equation*}
\begin{split}
k_{X}(x_{i},s_{1})&\geq k_{X}(x_{i},z)-k_{X}(s_{1},x_{1})-k_{X}(z,x_{1})\\
&\geq l_{k_{X}}(\beta_{i})-h-l_{k_{X}}(\alpha_{1})-l_{k_{X}}(\beta_{1}).
\end{split}
\end{equation*}
Since $l_{k_{X}}(\beta_{i})\rightarrow\infty\,\, (i\rightarrow\infty$),
we deduce that
$$
k_{X}(x_{i},s_{1})\rightarrow\infty \quad {\text {as}}\quad i\rightarrow\infty.
$$
Similarly,
$$
k_{X}(y_{i},s_{1})\rightarrow\infty \quad {\text {as}}\quad i\rightarrow\infty.
$$
This proves (2).

In what follows, it remains to prove (3) and (4).
Since
\begin{eqnarray*}
\begin{split}
0 &\leq (z|y_{i})_{x_{i}}=\frac{1}{2}\big(k_{X}(z,x_{i})+k_{X}(x_{i},y_{i})-k_{X}(z,y_{i})\big)\\
&\leq \frac{1}{2}\big[k_{X}(z,x_{i})+k_{X}(x_{i},y_{i})-(k_{X}(z,x_{i})-k_{X}(x_{i},y_{i}))\big]\\
&= k_{X}(x_{i},y_{i})\\
&\leq l_{k_{X}}(\alpha_{i}),
\end{split}
\end{eqnarray*}
we can find a point $w_{i}\in\alpha_{i}$ with $\alpha_{i}^{\prime}=\alpha_{i}|_{[x_{i},w_{i}]}$ such that
\begin{eqnarray*}
\begin{split}
l_{k_{X}}(\alpha_{i}^{\prime})=(z|y_{i})_{x_{i}}.
\end{split}
\end{eqnarray*}
Similarly, there exists a point $p_{i}\in\alpha_{i}$ with $\alpha_{i}^{\prime\prime}=\alpha_{i}|_{[p_{i},y_{i}]}$ such that
\begin{eqnarray*}
\begin{split}
l_{k_{X}}(\alpha_{i}^{\prime\prime})=(z|x_{i})_{y_{i}}.
\end{split}
\end{eqnarray*}
Set
$$
\alpha_{i}^{\ast}=\alpha_{i}|_{[w_{i}, p_{i}]}.
$$
Then we have $\alpha_{i}$ is the union of these three successive subarcs, and we say that such subdivision
\begin{equation*}\label{m-3}
\begin{split}
\alpha_{i}=\alpha_{i}^{\prime}\cup \alpha_{i}^{\ast}\cup \alpha_{i}^{\prime\prime}
\end{split}
\end{equation*}
is the subdivision of $\alpha_{i}$ induced by the triangle $\triangle_{i}$,
where $\triangle_{i}=(\beta_{i},\gamma_{i},\alpha_{i})$ is a quasihyperbolic $h$-short triangle for each positive integer $i$.
For the concepts of quasihyperbolic $h$-short triangle, please see \cite{V05, LLYT}.
For the other two sides $\beta_{i},\gamma_{i}$, we do a similar subdivision
\begin{equation*}\label{m-4}
\begin{split}
\beta_{i}=\beta_{i}^{\prime}\cup \beta_{i}^{\ast}\cup \beta_{i}^{\prime\prime}
\quad {\text {and}}\quad
\gamma_{i}=\gamma_{i}^{\prime}\cup \gamma_{i}^{\ast}\cup \gamma_{i}^{\prime\prime},
\end{split}
\end{equation*}
for each $i\in \mathbb{N}_+$, where
\begin{equation}\label{m-5}
\begin{split}
l_{k_{X}}(\beta_{i}^{\prime})=l_{k_{X}}(\gamma_{i}^{\prime})=(x_{i}|y_{i})_{z},
\end{split}
\end{equation}
\begin{equation}\label{m-6}
\begin{split}
l_{k_{X}}(\beta_{i}^{\prime\prime})=l_{k_{X}}(\alpha_{i}^{\prime})=(z|y_{i})_{x_{i}},
\end{split}
\end{equation}
\begin{equation}\label{m-7}
\begin{split}
l_{k_{X}}(\gamma_{i}^{\prime\prime})=l_{k_{X}}(\alpha_{i}^{\prime\prime})=(z|x_{i})_{y_{i}},
\end{split}
\end{equation}
and $z\in \beta_{i}^{\prime}\cap \gamma_{i}^{\prime}$,
$x_{i}\in\beta_{i}^{\prime\prime}$, $\gamma_{i}^{\prime\prime}=\gamma_{i}|_{[q_{i}, y_{i}]}$,
here $q_{i}\in\gamma_{i}^{\ast}\cap\gamma_{i}^{\prime\prime}$.

According to (\ref{m-6}), (\ref{m-7}) and the definition of quasihyperbolic $h$-short arc,
we get
\begin{equation}\label{m-8}
\begin{split}
l_{k_{X}}(\alpha_{i}^{\ast})&=l_{k_{X}}(\alpha_{i})-l_{k_{X}}(\alpha_{i}^{\prime})-l_{k_{X}}(\alpha_{i}^{\prime\prime})\\
&=l_{k_{X}}(\alpha_{i})-(z|y_{i})_{x_{i}}-(z|x_{i})_{y_{i}}\\
&=l_{k_{X}}(\alpha_{i})-k_{X}(x_{i},y_{i})\\
&\leq h.
\end{split}
\end{equation}
Using a similar argument as above inequality, we can obtain
\begin{equation}\label{m-9}
\begin{split}
l_{k_{X}}(\beta_{i}^{\ast})\leq h
\quad {\text {and}}\quad
l_{k_{X}}(\gamma_{i}^{\ast})\leq h.
\end{split}
\end{equation}
From (\ref{m-5}), (\ref{m-6}) and (\ref{m-9}), it follows that
\begin{equation*}\label{m-10}
\begin{split}
l_{k_{X}}(\beta_{i})\geq l_{k_{X}}(\beta_{i}^{\prime})+l_{k_{X}}(\beta_{i}^{\prime\prime})
=(x_{i}|y_{i})_{z}+l_{k_{X}}(\alpha_{i}^{\prime})
\end{split}
\end{equation*}
and
\begin{equation*}\label{m-11}
\begin{split}
l_{k_{X}}(\beta_{i})= l_{k_{X}}(\beta_{i}^{\prime})+ l_{k_{X}}(\beta_{i}^{\ast})+ l_{k_{X}}(\beta_{i}^{\prime\prime})
\leq (x_{i}|y_{i})_{z}+h+l_{k_{X}}(\alpha_{i}^{\prime}).
\end{split}
\end{equation*}
For $i<j$, combing the above estimate with (\ref{m-1}) and (\ref{m-2}), it follows immediately that
\begin{eqnarray*}
\begin{split}
3h&\leq (j-i)\cdot 3h\leq l_{k_{X}}(\beta_{j})-l_{k_{X}}(\beta_{i})\\
&\leq (x_{j}|y_{j})_{z}+h+l_{k_{X}}(\alpha_{j}^{\prime})-(x_{i}|y_{i})_{z}-l_{k_{X}}(\alpha_{i}^{\prime})\\
&\leq |(x_{j}|y_{j})_{z}-(x_{i}|y_{i})_{z}|+h+l_{k_{X}}(\alpha_{j}^{\prime})-l_{k_{X}}(\alpha_{i}^{\prime})\\
&\leq 2h+l_{k_{X}}(\alpha_{j}^{\prime})-l_{k_{X}}(\alpha_{i}^{\prime}).
\end{split}
\end{eqnarray*}
Thus, we have
\begin{eqnarray*}
\begin{split}
l_{k_{X}}(\alpha_{i}^{\prime})\leq l_{k_{X}}(\alpha_{j}^{\prime})-h.
\end{split}
\end{eqnarray*}
Using a similar argument as $\alpha_{i}^{\prime}$, we can obtain
\begin{eqnarray*}
\begin{split}
l_{k_{X}}(\alpha_{i}^{\prime\prime})\leq l_{k_{X}}(\alpha_{j}^{\prime\prime})-h.
\end{split}
\end{eqnarray*}
Therefore, by (\ref{m-8}), it follows from the above fact that
\begin{eqnarray*}
\begin{split}
l_{k_{X}}(\alpha_{i})-l_{k_{X}}(\alpha_{i}^{\prime})&=l_{k_{X}}(\alpha_{i}^{\ast})+l_{k_{X}}(\alpha_{i}^{\prime\prime})
\leq h+l_{k_{X}}(\alpha_{j}^{\prime\prime})-h\\
&=l_{k_{X}}(\alpha_{j}^{\prime\prime})
\leq l_{k_{X}}(\alpha_{j}^{\prime\prime})+l_{k_{X}}(\alpha_{j}^{\ast})\\
&= l_{k_{X}}(\alpha_{j})-l_{k_{X}}(\alpha_{j}^{\prime})\\
&\leq l_{k_{X}}(\alpha_{j})-l_{k_{X}}(\alpha_{i}^{\prime})-h.
\end{split}
\end{eqnarray*}
Since $h>0$, this guarantees that
\begin{eqnarray*}
\begin{split}
l_{k_{X}}(\alpha_{i})\leq l_{k_{X}}(\alpha_{j})-h<l_{k_{X}}(\alpha_{j}).
\end{split}
\end{eqnarray*}
Hence, it follows that there is a well defined orientation preserving length map
$f_{ij}:\alpha_{i}\rightarrow \alpha_{j}$ with $f_{ij}(w_{i})=w_{j}$ for $i\leq j$.
Let $u$ be any point in $\alpha_{i}$ and $i \leq j\leq m$,
it follows from the definition of length map that
\begin{eqnarray*}
\begin{split}
l_{k_{X}}\big(\alpha_{m}|_{[w_{m},f_{im}(u)]}\big)&=l_{k_{X}}\big(\alpha_{i}|_{[w_{i},u]}\big)=l_{k_{X}}\big(\alpha_{j}|_{[w_{j},f_{ij}(u)]}\big)\\
&=l_{k_{X}}\big(\alpha_{m}|_{[w_{m},f_{jm}(f_{ij}(u))]}\big)\\
&=l_{k_{X}}\big(\alpha_{m}|_{[w_{m},(f_{jm}\circ f_{ij})(u)]}\big),
\end{split}
\end{eqnarray*}
which implies that $f_{im}(u)=(f_{jm}\circ f_{ij})(u)$.
Since $u$ is arbitrary, we obtain that  $f_{ii}=id$ and $f_{im}=f_{jm}\circ f_{ij}$ for $i \leq j\leq m$. Thus (3) is proved.

Finally, we only need to show (4). For all $u\in \alpha_i$,
where $\alpha_{i}=\alpha_{i}^{\prime}\cup \alpha_{i}^{\ast}\cup \alpha_{i}^{\prime\prime}$.
Now, according to the position of $u$ in $\alpha_{i}$, we divide the discussions into three cases:

\vspace{0.1cm}
\noindent {\bf {Case 1.}}\,\, $u\in \alpha_{i}^{\prime\prime}.$
\vspace{0.1cm}

Since $l_{k_{X}}(\gamma_{i}^{\prime\prime})=l_{k_{X}}(\alpha_{i}^{\prime\prime})$,
there is a bijective length map
$\varphi_{i}:\alpha_{i}^{\prime\prime}\rightarrow \gamma_{i}^{\prime\prime}$
with $\varphi_{i}(y_{i})=y_{i}$ for all $i$.
Thus, we have
\begin{equation}\label{m-12}
\begin{split}
l_{k_{X}}\big(\alpha_{i}^{\prime\prime}|_{[u, y_{i}]}\big)=l_{k_{X}}\big(\gamma_{i}^{\prime\prime}|_{[\varphi_{i}(u), y_{i}]}\big).
\end{split}
\end{equation}

In addition,
since $\gamma_{i}^{\prime\prime}=\gamma_{i}|_{[q_{i}, y_{i}]}, \alpha_{i}^{\prime\prime}=\alpha_{i}|_{[p_{i}, y_{i}]}$,
where $p_{i}\in\alpha_{i}^{\prime\prime}$ and $q_{i}\in\gamma_{i}^{\ast}\cap\gamma_{i}^{\prime\prime}$,
we obtain that
$$
l_{k_{X}}\big(\gamma_{i}^{\prime\prime}|_{[q_{i}, y_{i}]}\big)=l_{k_{X}}\big(\alpha_{i}^{\prime\prime}|_{[p_{i}, y_{i}]}\big)=l_{k_{X}}\big(\gamma_{i}^{\prime\prime}|_{[\varphi_{i}(p_{i}), y_{i}]}\big),
$$
and then it follows that
$$
\varphi_{i}(p_{i})=q_{i}.
$$

From (\ref{m-12}), it is obvious that
\begin{eqnarray*}
\begin{split}
l_{k_{X}}\big(\alpha_{i}|_{[u, y_{i}]}\big)=l_{k_{X}}\big(\gamma_{i}|_{[\varphi_{i}(u), y_{i}]}\big).
\end{split}
\end{eqnarray*}
In view of the fact
$$
k_{X}(u,y_{i})\leq l_{k_{X}}\big(\alpha_{i}^{\prime\prime}|_{[u, y_{i}]}\big)\leq l_{k_{X}}(\alpha_{i}^{\prime\prime})=(z|x_{i})_{y_{i}}
$$
and the process of proving the conclusion (3) of Lemma \ref{le-1}, we have
\begin{equation}\label{m-13}
\begin{split}
k_{X}\big(\varphi_{i}(u),u\big)\leq 4\delta+2h.
\end{split}
\end{equation}
Since $\varphi_{i}(u)\in \gamma_{i}^{\prime\prime}\subset\gamma_{i}$, it follows from (\ref{m-0}) that
\begin{equation}\label{m-15}
\begin{split}
k_{X}\Big(g_{ij}\big(\varphi_{i}(u)\big), \varphi_{i}(u)\Big)\leq 4\delta+2h.
\end{split}
\end{equation}

Together with (\ref{m-5}) and (\ref{m-12}),
it follows immediately from the definition of length map $g_{ij}$ that
\begin{equation}\label{m-16}
\begin{split}
l_{k_{X}}\big(\gamma_{j}|_{[z, g_{ij}(\varphi_{i}(u))]}\big)&=l_{k_{X}}\big(\gamma_{i}|_{[z, \varphi_{i}(u)]}\big)\\
&=l_{k_{X}}\big(\gamma_{i}^{\prime}\big)+l_{k_{X}}(\gamma_{i}^{\ast})+ l_{k_{X}}\big(\gamma_{i}^{\prime\prime}|_{[q_{i}, \varphi_{i}(u)]}\big) \\
&=(x_{i}|y_{i})_{z}+l_{k_{X}}(\gamma_{i}^{\ast})+l_{k_{X}}\big(\alpha_{i}^{\prime\prime}|_{[p_{i},u]}\big).
\end{split}
\end{equation}

As $g_{ij}(\varphi_{i}(u))\in \gamma_{j}$,
We now also consider two subcases:

\vspace{0.1cm}
\noindent {\bf {Subcase 1.1.}}\,\, $g_{ij}\big(\varphi_{i}(u)\big)\in \gamma_{j}^{\prime\prime}.$
\vspace{0.1cm}

Since $\varphi_{j}:\alpha_{j}^{\prime\prime}\rightarrow \gamma_{j}^{\prime\prime}$ is a bijective length map,
there is $w\in \alpha_{j}^{\prime\prime}$ with $\varphi_{j}(w)=g_{ij}(\varphi_{i}(u))$.
From (\ref{m-13}) and (\ref{m-15}), it is clear that
\begin{equation}\label{m-17}
\begin{split}
k_{X}(u,w)\leq k_{X}\big(u,\varphi_{i}(u)\big)+k_{X}\big(\varphi_{i}(u),g_{ij}(\varphi_{i}(u))\big)+k_{X}\big(g_{ij}(\varphi_{i}(u)),w\big)\leq 3\cdot(4\delta+2h).
\end{split}
\end{equation}
Combing (\ref{m-5}), (\ref{m-12}) and (\ref{m-16}), it follows that
\begin{eqnarray*}
\begin{split}
l_{k_{X}}\big(\alpha_{j}|_{[w_{j},w]}\big)&=l_{k_{X}}(\alpha_{j}^{\ast})+l_{k_{X}}\big(\alpha_{j}^{\prime\prime}|_{[p_{j},w]}\big)\\
&=l_{k_{X}}(\alpha_{j}^{\ast})+l_{k_{X}}\big(\gamma_{j}^{\prime\prime}|_{[\varphi_{j}(p_{j}), \varphi_{j}(w)]}\big)\\
&=l_{k_{X}}(\alpha_{j}^{\ast})+l_{k_{X}}\big(\gamma_{j}^{\prime\prime}|_{[q_{j}, g_{ij}(\varphi_{i}(u))]}\big)\\
&=l_{k_{X}}(\alpha_{j}^{\ast})+l_{k_{X}}\big(\gamma_{j}|_{[z,g_{ij}(\varphi_{i}(u))]}\big)-l_{k_{X}}(\gamma_{j}^{\prime})-l_{k_{X}}(\gamma_{j}^{\ast})\\
&=l_{k_{X}}(\alpha_{j}^{\ast})+(x_{i}|y_{i})_{z}-(x_{j}|y_{j})_{z}+l_{k_{X}}(\gamma_{i}^{\ast})
  -l_{k_{X}}(\gamma_{j}^{\ast})+l_{k_{X}}\big(\alpha_{i}^{\prime\prime}|_{[p_{i},u]}\big),
\end{split}
\end{eqnarray*}
and it follows from the definition of length map $f_{ij}$ that
\begin{eqnarray*}
\begin{split}
l_{k_{X}}\big(\alpha_{j}|_{[w_{j},f_{ij}(u)]}\big)=l_{k_{X}}\big(\alpha_{i}|_{[w_{i},u]}\big)
 =l_{k_{X}}(\alpha_{i}^{\ast})+l_{k_{X}}\big(\alpha_{i}^{\prime\prime}|_{[p_{i},u]}\big).
\end{split}
\end{eqnarray*}
Therefore, according to (\ref{m-1}), (\ref{m-8}) and (\ref{m-9}), it follows immediately that
\begin{equation}\label{m-18}
\begin{split}
k_{X}(w,f_{ij}(u))&\leq l_{k_{X}}\big(\alpha_{j}|_{[w,f_{ij}(u)]}\big)\\
&\leq\big|l_{k_{X}}\big(\alpha_{j}|_{[w_{j},w]}\big)-l_{k_{X}}\big(\alpha_{j}|_{[w_{j},f_{ij}(u)]}\big)\big|\\
&=\big|l_{k_{X}}(\alpha_{j}^{\ast})-l_{k_{X}}(\alpha_{i}^{\ast})+(x_{i}|y_{i})_{z}
 -(x_{j}|y_{j})_{z}+l_{k_{X}}(\gamma_{i}^{\ast})-l_{k_{X}}(\gamma_{j}^{\ast})\big|\\
&\leq l_{k_{X}}(\alpha_{j}^{\ast})+\big|(x_{i}|y_{i})_{z}-(x_{j}|y_{j})_{z}\big|+l_{k_{X}}(\gamma_{i}^{\ast})\\
&\leq 3h.
\end{split}
\end{equation}
Thus, from (\ref{m-17}) and (\ref{m-18}), we deduce that
\begin{eqnarray*}
\begin{split}
k_{X}\big(f_{ij}(u), u\big)\leq k_{X}(u,w)+k_{X}(w,f_{ij}(u))\leq 12\delta+9h.
\end{split}
\end{eqnarray*}

\vspace{0.1cm}
\noindent {\bf {Subcase 1.2.}}\,\, $g_{ij}(\varphi_{i}(u))\notin \gamma_{j}^{\prime\prime}.$
\vspace{0.1cm}

It follows from (\ref{m-5}) and (\ref{m-16}) that
\begin{eqnarray*}
\begin{split}
l_{k_{X}}\big(\gamma_{j}|_{[g_{ij}(\varphi_{i}(u)),q_{j}]}\big)&=l_{k_{X}}\big(\gamma_{j}|_{[z,q_{j}]}\big)-l_{k_{X}}\big(\gamma_{j}|_{[z,g_{ij}(\varphi_{i}(u))]}\big)\\
&=l_{k_{X}}(\gamma_{j}^{\prime})+l_{k_{X}}(\gamma_{j}^{\ast})-l_{k_{X}}\big(\gamma_{j}|_{[z,g_{ij}(\varphi_{i}(u))]}\big)\\
&=(x_{j}|y_{j})_{z}+l_{k_{X}}(\gamma_{j}^{\ast})-(x_{i}|y_{i})_{z}-l_{k_{X}}(\gamma_{i}^{\ast})-l_{k_{X}}\big(\alpha_{i}^{\prime\prime}|_{[p_{i},u]}\big).
\end{split}
\end{eqnarray*}
Combining this equation with (\ref{m-1}) and (\ref{m-9}), we now have
\begin{equation}\label{m-19}
\begin{split}
l_{k_{X}}\big(\gamma_{j}|_{[g_{ij}(\varphi_{i}(u)),q_{j}]}\big)+l_{k_{X}}\big(\alpha_{i}^{\prime\prime}|_{[p_{i},u]}\big)\leq \big|(x_{j}|y_{j})_{z}-(x_{i}|y_{i})_{z}\big|+l_{k_{X}}(\gamma_{j}^{\ast})\leq2h.
\end{split}
\end{equation}
Moreover, according to (\ref{m-8}) and the definition of length map $f_{ij}$, we get
\begin{equation}\label{m-20}
\begin{split}
k_{X}\big(f_{ij}(u),p_{j}\big)&\leq k_{X}\big(w_{j},f_{ij}(u)\big)+k_{X}(w_{j},p_{j})\\
&\leq l_{k_{X}}\big(\alpha_{j}|_{[w_{j},f_{ij}(u)]}\big)+l_{k_{X}}(\alpha_{j}^{\ast})\\
&=l_{k_{X}}\big(\alpha_{i}|_{[w_{i},u]}\big)+l_{k_{X}}(\alpha_{j}^{\ast})\\
&= l_{k_{X}}(\alpha_{i}^{\ast})+l_{k_{X}}\big(\alpha_{i}^{\prime\prime}|_{[p_{i},u]}\big)+l_{k_{X}}(\alpha_{j}^{\ast})\\
&\leq l_{k_{X}}\big(\alpha_{i}^{\prime\prime}|_{[p_{i},u]}\big)+2h.
\end{split}
\end{equation}
In addition, remembering that $p_{j}\in\alpha_{j}^{\prime\prime}$, $\varphi_{j}(p_{j})=q_{j}$, (\ref{m-13}) and (\ref{m-15}) implies that
$$
k_{X}(p_{j},q_{j})\leq 4\delta+2h
$$
and
\begin{eqnarray*}
\begin{split}
k_{X}\big(q_{j},\varphi_{i}(u)\big)&\leq k_{X}\big(q_{j},g_{ij}(\varphi_{i}(u))\big)+k_{X}\big(g_{ij}(\varphi_{i}(u)),\varphi_{i}(u)\big)\\
&\leq l_{k_{X}}\big(\gamma_{j}|_{[g_{ij}(\varphi_{i}(u)),q_{j}]}\big)+4\delta+2h.
\end{split}
\end{eqnarray*}
Therefore, according to (\ref{m-13}), (\ref{m-19}) and (\ref{m-20}), it follows from the above fact that
\begin{eqnarray*}
\begin{split}
k_{X}\big(f_{ij}(u),u\big)&\leq k_{X}\big(f_{ij}(u),p_{j}\big)+k_{X}(p_{j},q_{j})+k_{X}\big(q_{j},\varphi_{i}(u)\big)+k_{X}(\varphi_{i}(u),u)\\
&\leq \big(l_{k_{X}}\big(\alpha_{i}^{\prime\prime}|_{[p_{i},u]}\big)+2h\big)+(4\delta+2h)
 +\big(l_{k_{X}}\big(\gamma_{j}|_{[g_{ij}(\varphi_{i}(u)),q_{j}]}\big)+4\delta+2h\big)+(4\delta+2h)\\
&= l_{k_{X}}\big(\alpha_{i}^{\prime\prime}|_{[p_{i},u]}\big)+l_{k_{X}}\big(\gamma_{j}|_{[g_{ij}(\varphi_{i}(u)),q_{j}]}\big)+12\delta+8h\\
&\leq 12\delta+10h.
\end{split}
\end{eqnarray*}

\vspace{0.1cm}
\noindent {\bf {Case 2.}}\,\, $u\in \alpha_{i}^{\prime}.$
\vspace{0.1cm}

Since $l_{k_{X}}(\alpha_{i}^{\prime})=l_{k_{X}}(\beta_{i}^{\prime\prime})$,
there is a bijective length map $\psi_{i}:\alpha_{i}^{\prime}\rightarrow \beta_{i}^{\prime\prime}$ with $\psi_{i}(x_{i})=x_{i}$ for all $i$.
The proof of this case is similar to Case 1.
The major change is the substitution of $h_{ij}$ and $\psi_{i}$ for $g_{ij}$ and $\varphi_{i}$, respectively.
Hence, for all $u\in \alpha_{i}^{\prime}$, we can obtain
$$
k_{X}\big(f_{ij}(u),u\big)\leq 12\delta+10h.
$$

\noindent {\bf {Case 3.}}\,\, $u\in \alpha_{i}^{\ast} .$

As $w_{i}\in \alpha_{i}^{\prime}$ , according to Case 2, we obtain that
\begin{equation}\label{m-21}
\begin{split}
k_{X}(w_{j},w_{i})=k_{X}\big(f_{ij}(w_{i}),w_{i}\big)\leq 12\delta+10h.
\end{split}
\end{equation}
For all $u\in \alpha_{i}^{\ast}$, by the definition of length map $f_{ij}$, (\ref{m-8}) and (\ref{m-21}),
it follows immediately that
\begin{equation*}\label{m-22}
\begin{split}
k_{X}\big(f_{ij}(u),u\big)
&\leq k_{X}\big(f_{ij}(u),w_{j}\big)+k_{X}(w_{j},w_{i})+k_{X}(w_{i},u)\\
&\leq l_{k_{X}}\big(\alpha_{j}|_{[f_{ij}(u), w_{j}]}\big)+12\delta+10h+l_{k_{X}}(\alpha_{i}^{\ast})\\
&=l_{k_{X}}\big(\alpha_{i}|_{[u, w_{i}]}\big)+12\delta+10h+l_{k_{X}}(\alpha_{i}^{\ast})\\
&\leq 2\cdot l_{k_{X}}(\alpha_{i}^{\ast})+12\delta+10h\\
&\leq 12\delta+12h,
\end{split}
\end{equation*}
as desired.

Combining the above three cases, it can be seen that $f_{ij}$ satisfies
$$
k_{X}\big(f_{ij}(u), u\big)\leq 12(\delta+h)
$$
for all $u\in \alpha_{i}$. This yields (4).

Hence, Theorem \ref{main} is proved.
$\hfill\Box$

\vspace{0.1cm}

\bigskip \noindent{\bf Acknowledgement}.
We would like to express our deep gratitude to the referee for his
or her careful reading and useful suggestions and remarks on this paper.

\bibliographystyle{amsplain}

\end{document}